\documentclass[12pt]{article}

\usepackage[top=2cm,left=2cm,right=2cm,bottom=3cm]{geometry}
\usepackage{amsfonts}
\usepackage{amssymb}
\usepackage[english]{babel}
\usepackage{fancyhdr}
\usepackage[fleqn]{amsmath}
\usepackage{latexsym}
\usepackage{graphicx}
\usepackage{makeidx}

\DeclareMathOperator{\lcm}{lcm}

\newtheorem{theorem}{Theorem}
\newtheorem{lemma}[theorem]{Lemma}
\newtheorem{corollary}[theorem]{Corollary}

\newenvironment{proof}{\textbf{Proof}}{\mbox{}\hfill$\Box$\par}

\begin{document}
\title {Cycles and divergent trajectories  \\ for a class  of  permutation sequences}
\author{John L Simons \thanks {University of Groningen, PO Box 800, 9700 AV Groningen, The
Netherlands. Email: j.l.simons@rug.nl}}       
\maketitle

\subsection*{Abstract}
Let $f$ be a permutation from $\mathbb{N}_0$ onto  $\mathbb{N}_0$. Let $x\in\mathbb{N}_0$ and consider a (finite or infinite) sequence
$s= (x,f(x),f^2(x),\cdots)$. We call $s$ a permutation sequence. Let $D$ be the set of elements of $s$. If $D$ is a finite set then the sequence $s$ is a cycle, and if $D$ is an infinite set the sequence $s$ is a divergent trajectory.  We derive theoretical and computational bounds for cycles and divergent trajectories for a defined class of permutations.

\subsection*{Mathematics Subject Classifications} 11B83,\ 11J86

\section{Introduction}
\subsection{Examples}
Lothar Collatz (the inventor of the famous Collatz conjecture) studied the permutation
 $f(2n)=3n,\ f(4n+1)=3n+1,\ f(4n+3)=3n+2$  with $n \in \mathbb{N}_0$. See the section \textbf{Remarks $1$}.
There is computational evidence (no proof) that the only cycles are $(0),\ (1),\ (2,3),\ (4,6,9,7,5),\\ (44,66,99,74,111,83,62,93,70,105,79,59)$ and that the numbers $8,\ 14,\ 40,\ 64,\ 80,\ 82,\cdots $ are minima of different divergent trajectories e.g.\\ $ \cdots,97,73,55,41,31,23,17,13,10,15,11,8,12,18,27,20,30,45,34,51,38,57,\cdots$. We call this function the  Collatz permutation and the associated sequences Collatz permutation sequences.\\

Consider the function $f(2n)=n$ for $n>0$,\ $f(0)=1$,\ $f(2n+1)=2n+3$ for $n\ge 0$. Now only the double divergent trajectory $ \cdots,16,8,4,2,0,1,3,5,7,\cdots $ exists.

\subsection{Definitions}
We denote the set of positive integers by $\mathbb{N}$ and the set of non-negative integers by $\mathbb{N}_0$. Let $f\ :\ \mathbb{N}_0 \rightarrow \mathbb{N}_0$ be a permutation (bijection).
Let $x_0\in\mathbb{N}_0$. Consider a sequence $(x_0,x_1=f(x_0),x_2=f(x_1)=f^2(x_0),\cdots)$ associated with $f$. Then either there exists a smallest $k\in\mathbb{N}$ such that $x_k=x_0$ or such a $k\in\mathbb{N}$ does not exist. If $k$ exists then  $x_0$ is an element of a cycle of length $k$. If $k$ does not exist, then $x_0$ is an element of a divergent trajectory with $\lim_{k \rightarrow \infty}f^k(x_0)=\infty$ and also $\lim_{k \rightarrow \infty}f^{-k}(x_0)=\infty$.  Guy \cite{Gu} (problem $E17$) calls this a double infinite chain.
Let $D$ be the set of (different) $x_j$-values of such a sequence. Then $f$ is also a permutation (bijection) from $D$ onto $D$.
We call sequences that are associated with $f$ permutation sequences. We denote the set of these sequences by $PS(f)$. \\

Let $a,b,c,d \in \mathbb{N}$.  Assume that $b>1,d>1$ and $N=a(b-1)=c(d-1)$.
Let $R = \{r_i > 0 |\ i=1 \cdots N\}$ be the set of different residues (mod $ab$) with $r_i \not\equiv 0$ (mod $b$) and let $S = \{s_i > 0 |\ i=1 \cdots N\}$ be the set of different residues (mod $cd$) with $s_i \not\equiv 0$ (mod $d$).

\begin{lemma}\label{gps}
Let $a,b,c,d,r_i,s_i$ be defined as above.
Let $n \in \mathbb{N}_0$. Let $f$ be a function from $\mathbb{N}_0$ onto $\mathbb{N}_0$, defined by $f(bn)=dn,\ f(abn+r_i)=cdn+s_i|i=1,\cdots,N$. Then $f\ :\ \mathbb{N}_0 \rightarrow \mathbb{N}_0$  is a bijection (permutation).
\end{lemma}
\begin{proof}
Let $x \in \mathbb{N}_0$. Then either $x \equiv 0$\ ( mod\ $b$) or there exists exactly one $i$ with $x \equiv r_i$\ (mod\ $ab$). Hence $f$ is a surjective function from $\mathbb{N}_0$ onto $\mathbb{N}_0$.
Let $g$ be a function from $\mathbb{N}_0$ onto $\mathbb{N}_0$, defined by $g(dn)=bn,\ g(cdn+s_i)=abn+r_i|i=1,\cdots,N$.  Similarly $g$ is a surjective function from $\mathbb{N}_0$ onto $\mathbb{N}_0$.
Since $g$ is the inverse function of $f$, $f$ is a permutation (bijection) from $\mathbb{N}_0$ onto $\mathbb{N}_0$.
\end{proof}
\ \\
Lemma \ref{gps} is independent of the order of the residues. We arbitrarily choose $1=r_1<r_2< \cdots <r_N=ab-1$. If also $1=s_1<s_2< \cdots<s_N=cd-1$ we write $f$  as $P(a,b,c,d)$.  We write $P(a,b,c,d)$ in the form
\[
\left.
\begin{array}{c}
bn \\
abn+r_1\\
\cdots \\
abn+r_N
\end{array}
\right\}
\leftrightarrow
\left\{
\begin{array} {c}
dn\\
cdn+s_1\\
\cdots \\
cdn+s_N
\end{array}
\right.
 \]

If the residues are in non-increasing order we call such $f$ a simple generalization of $P(a,b,c,d)$.
In this paper we initially consider the class of permutations consisting of $P(a,b,c,d)$ and their simple generalizations. This class will be extended later.\\

The first example in the subsection \textbf{Examples} i.e. the Collatz permutation is $P(2,2,1,3)$. Another example of a permutation in the given format is $P(2,4,3,3)$. 

\[
\left.
\begin{array}{c}
4n \\
8n+1\\
8n+2\\
8n+3\\
8n+5\\
8n+6\\
8n+7
\end{array}
\right\}
\leftrightarrow
\left\{
\begin{array} {c}
3n\\
9n+1\\
9n+2\\
9n+4\\
9n+5\\
9n+7\\
9n+8
\end{array}
\right.
 \]

For both examples we will calculate bounds for cycles and divergent trajectories. Note that the second example in the subsection \textbf{Examples} cannot be written in the format $P(a,b,c,d)$.

\subsection{Main result}
For the original Collatz function $g(2n)=n,\ g(2n+1)=3n+2$ (which is not a permutation), a trajectory with $K$ odd and $L$ even elements, $m$ local minima, $m$ local maxima, and with the property that for all elements $x_i$ of the trajectory  $x_{i+K+L}=x_i$ is called an $m$-cycle by Simons and de Weger \cite{SW}. Steiner and also Davidson \cite{St,Da} call a $1$-cycle a circuit. Brox \cite{Br} considers cycles of odd numbers only and calls a number a descendent if the next number is smaller. A descendent is the (odd) predecessor of a local maximum in an $m$-cycle.
For the Collatz function Steiner \cite{St} proved that $1$-cycles, except $(1,2)$, cannot exist. Simons and de Weger generalized his proof to $m$-cycles up to $m=75$, however  their approach is not simply applicable to permutation sequences. We generalize the approach of Simons and de Weger in a nontrivial way. For permutation sequences we modify the definition of an $m$-cycle to : a trajectory with $K$ elements $\not\equiv 0$ (mod $b$) and $L$ elements  $\equiv 0$  (mod $b$), $m$ local minima and $m$ local maxima and with the property that for all elements $x_i$ of the trajectory  $x_{i+K+L}=x_i$.
Our main result is

\begin{theorem}[Main Theorem] \label{mth}
\
\begin{enumerate}
  \item For all $a,b,c,d$, if $f=P(a,b,c,d)$ then  $PS(f)$ has for each $m$ a finite number of $m$-cycles. These cycles can be computed.
  \item For all $a,b,c,d$, the set of sequences associated with a simple generalization of $P(a,b,c,d)$ has for each $m$ a finite number of $m$-cycles. These cycles can be computed.
  \item The set of Collatz permutation sequences associated with $P(1,3,2,2)$ has for $m \le 5$ no other cycles than those listed  in Corollary \ref{cycinvcol}.
  \item The set of permutation sequences associated with the simple generalization of $P(1,3,2,2)$ has for $m \le 5$ no other cycles than those listed  in the subsection \textbf{Simple generalizations}.
  \item If $f=P(2,4,3,3)$ then $PS(f)$ has for  $m\le 5 $ no other cycles than those listed in Corollary \ref{cycps2433}.
  \item There exist permutations $f=P(a,b,c,d)$ such that the set of associated permutation sequences $PS(f)$ contains an infinite number of divergent trajectories.
\end{enumerate}
\end{theorem}
\ \\

\section{Conditions for the existence of $m$-cycles}

Consider $P(a,b,c,d)$ with $a>0,\ c>0,\ b>1,\ d>1$. We choose  $b>d$ which implies $c>a,\ ab<cd$.  Then an $m$-cycle consists of $m$ pairs of an increasing  subsequence  of $k_i$ elements $\not\equiv 0$ (mod $b$)  followed by a decreasing subsequence of $\ell_i$ elements $\equiv 0 $ (mod $b$).

Suppose an $m$-cycle exists for $P(a,b,c,d)$ with elements $x_i>X_0=\frac{1+\epsilon}{\epsilon}(ab-a-1) \gg ab-a-1$ where $X_0$ is a numerical lower bound. By convention this $m$-cycle has $m$ pairs of an increasing subsequence followed by a decreasing subsequence. The elements in the $m$-cycle are denoted by $x_i$. We call the local minima $\overline{x}_0 \dots \overline{x}_{m-1}$, the local maxima $\overline{y}_0 \dots \overline{y}_{m-1}$ and the length of the subsequences $k_0 \ldots k_{m-1}$ and $\ell_0 \ldots \ell_{m-1}$ respectively. We put $K=\sum_{j=0}^{m-1}k_j$ and $L=\sum_{j=0}^{m-1}\ell_j$, hence $x_{K+L}=x_0$. For $m$-cycles we define $\Lambda=K\log(cd)-K\log(ab)+L\log(d)-L\log(b)$.

\subsection{An upper bound for $\Lambda$ in terms of the local minima $\overline{x}_j$}
For the first pair of subsequences (neglecting the indices of $r$ and $s$) we have
 \[
 \frac{x_1}{x_0}=\frac{cd.n_0+s}{ab.n_0+r}, \cdots,\ \frac{x_{k_0}}{x_{k_0-1}}=\frac{cd.n_{k_0-1}+s}{ab.n_{k_0-1}+r},\ \frac{x_{k_0+1}}{x_{k_0}}= \cdots =\frac{x_{k_0+\ell_0}}{x_{k_0+\ell_0-1}}=\frac{d}{b}
 \]

Multiplication of all chaining relations an using $x_{k_0+\ell_0}=\overline{x}_1$ yields
\[
\frac{\overline{x}_1}{\overline{x}_0}=\frac{x_{k_0+\ell_0}}{x_0}=\frac{x_1}{x_0} \frac{x_2}{x_1} \cdots \frac{x_{k_0+\ell_0}}{x_{k+\ell_0-1}}=\left(\frac{cd}{ab}\right)^{k_0}\left(\frac{d}{b}\right)^{\ell_0}\prod_{i=0}^{k_0-1}\frac{1+\frac{s}{cdn_i}}{1+\frac{r}{abn_i}}
\]

Applying this to each pair of subsequences  we find for $j=0 \dots m-1$
\[
\frac{\overline{x}_{j+1}}{\overline{x}_j}=\left(\frac{cd}{ab}\right)^{k_j}\left(\frac{d}{b}\right)^{\ell_j}Pr(j)
\]
where $Pr(j)=\prod_{i=0}^{k_j-1}\frac{1+\frac{s}{cdn_i}}{1+\frac{r}{abn_i}}$.
Multiplication of all the fractions of local minima and using $\overline{x}_{m}=\overline{x}_{0}$ leads to
\[
F(K,L)=\left(\frac{cd}{ab}\right)^{K}\left(\frac{d}{b}\right)^{L}=\prod_{j=0}^{m-1}(Pr(j))^{-1}
\]
Suppose $F(K,L)>1$. Then we obtain
\[
1<F(K,L)<\prod_{j=0}^{m-1}\prod_{i=0}^{k_j-1}[1+\frac{r}{abn_i}]<\prod_{j=0}^{m-1}\left(1+\frac{ab-a-1}{\overline{x}_j-ab+a+1}\right)^{k_j}
\]
 which implies after taking logs

\begin{equation}\label{Laub1}
0<\Lambda=\log F(K,L) < \sum_{j=0}^{m-
1}k_j\log\left(1+\frac{ab-a-1}{\overline{x}_j-ab+a+1}\right)
<\sum_{j=0}^{m-1}k_j\frac{ab-a-1}{\overline{x}_j-ab+a+1}
\end{equation}
Suppose $F(K,L)<1$. Then we find similarly

\begin{equation}\label{Laub2}
0<-\Lambda=\log (F(K,L))^{-1} < \sum_{j=0}^{m-1}k_j\log\left(1+\frac{cd-c-1}{\overline{x}_j-cd+c+1}\right)<\sum_{j=0}^{m-1}k_j\frac{cd-c-1}{\overline{x}_j-cd+c+1}
\end{equation}
From equations \ref{Laub1} and \ref{Laub2} and using the lower bound $\overline{x}_j > X_0$ we find
\begin{corollary}\label{LaubK}\ \\
If $F(K,L)>1$ then $0<\Lambda< K \frac{ab-a-1}{X_0-ab+a+1}=\epsilon K$\\
If $F(K,L)<1$ then $0<-\Lambda< K \frac{ab-a-1}{X_0-ab+a+1}=\epsilon K$
\end{corollary}

\subsection{An upper bound for $\Lambda$ in terms of $K$,\ $L$ and $m$}
For local minima we have by convention $\overline{x}_j= e_j.d^{\ell_j}\ge  d^{\ell_j}$. For the maximal local minimum we find $\max(\overline{x}_j)^m > \prod_{j=0}^{m-1}\overline{x}_j>d^L$ and consequently

\begin{equation}\label{UBminima}
  \max(\overline{x}_j) > d^{\frac{L}{m}}
\end{equation}

We now derive a chaining relation between the magnitudes of successive local minima $\overline{x}_j$.
Under the assumption $a<c,b>d,ab<cd$, consider the trajectory from $\overline{x}_j$ up to $\overline{y}_j$ and down to $\overline{x}_{j+1}$.
Then $\overline{x}_j$ must be less than or equal to the largest possible predecessor of $\overline{y}_j$. The largest possible predecessor occurs if $abn+ab-a-1 \leftrightarrow cdn+1$ applies to that predecessor. Let $\overline{y}_j= e_jb^{\ell_j}$, $\delta_1=\log_d b+1 (>1)$, $\alpha_1= -\frac{ab}{cd}+ab-a-1 (>0)$. Then we find the asserted relation
\begin{eqnarray*}
  \overline{x}_j &\le& abn+ab-a-1 = \frac{ab}{cd}(cdn+1) - \frac{ab}{cd}+ab-a-1 = \frac{ab}{cd} \overline{y}_j +\alpha_1\\
   &=& \frac{ab}{cd} e_jb^{\ell_j}+ \alpha_1 = \frac{ab}{cd} \left(\frac{b}{d}\right)^{\ell_j}e_jd^{\ell_j}+ \alpha_1= \frac{ab}{cd} (d^{\delta_1})^{\ell_j}e_jd^{\ell_j}+\alpha_1  \\
   &=& \frac{ab}{cd}\frac{ (e_jd^{\ell_j})^{\delta_1+1}}{e_j^{\delta_1}}+\alpha_1 < \frac{ab}{cd}\overline{x}_{j+1}^{\delta_1+1}+\alpha_1
\end{eqnarray*}
\ \\
Let $\gamma_1= \frac{ab}{cd} + \frac{\alpha_1}{X_0^{\delta_1+1}}$. Then $\alpha_1=(\gamma_1 -\frac{ab}{cd})X_0^{\delta_1+1}< (\gamma_1 -\frac{ab}{cd})\overline{x}_{j+1}^{\delta_1+1}$  which implies
$ \overline{x}_j < \frac{ab}{cd}\overline{x}_{j+1}^{\delta_1+1}+(\gamma_1 -\frac{ab}{cd})\overline{x}_{j+1}^{\delta_1+1} = \gamma_1 \overline{x}_{j+1}^{\delta_1+1}$.
We find (taking into account that the worst case applies if $\overline{x}_j>\overline{x}_{j+1}$ for $j=0 \cdots m-2$)

\begin{equation}\label{chainmaxmin}
\max (\overline{x}_j)< \gamma_1^{1+ (\delta_1+1)+ \dots + (\delta_1+1)^{m-2}} \left(\min (\overline{x}_j)\right)^{(\delta_1+1)^{m-1}}
\end{equation}

Let $\beta_1 = \gamma_1^{1+ (\delta_1+1)+ \dots + (\delta_1+1)^{m-2}} = \gamma_1^\frac{(\delta_1+1)^{m-1}-1}{\delta_1} $. Inserting equations  \ref{UBminima} and \ref{chainmaxmin} into equation \ref{Laub1} we obtain an upper bound for $\Lambda$ (if $\Lambda>0$) in terms of $K$, $L$ and $m$.
\begin{eqnarray*}\label{LaUBKLm}
  0<\Lambda=\log F(K,L) &<&  \sum_{j=0}^{m-1}k_j\frac{ab-a-1}{\overline{x}_j-ab+a+1} < K\frac{ab-a-1}{\min(\overline{x}_j)-ab+a+1}\\
  &=& K\frac{(1+\epsilon)(ab-a-1)}{\min(\overline{x}_j)}\frac{\min(\overline{x}_j)}{\min(\overline{x}_j)-ab+a+1}\frac{1}{1+\epsilon}\\
  &<& K\frac{(1+\epsilon)(ab-a-1)}{\min(\overline{x}_j)}\frac{X_0}{X_0-ab+a+1}\frac{1}{1+\epsilon}\\
  &=& K\frac{(1+\epsilon)(ab-a-1)}{\min(\overline{x}_j)}\\
  &<& K \frac{(1+\epsilon)(ab-a-1)}{\left(\frac{\max(\overline{x}_j)}{\beta_1}\right)^{\frac{1}{(\delta_1+1)^{m-1}}}}
  < K \frac{(1+\epsilon)(ab-a-1)}{\left(\frac{d^{\frac{L}{m}}}{\beta_1}\right)^{\frac{1}{(\delta_1+1)^{m-1}}}}
\end{eqnarray*}
A similar expression  can be derived if $\Lambda<0$. So we have
\begin{corollary}\label{UbKLm} \ \\
If an $m$-cycle exists for $P(a,b,c,d)$ with $0<a<c,\ b>d>0$ and $\epsilon= \frac{ab-a-1}{X_0-(ab-a-1)}$ then
 $ |\Lambda| < K \frac{(1+\epsilon)(ab-a-1)}{\left(\frac{d^{\frac{L}{m}}}{\beta_1}\right)^{\frac{1}{(\delta_1+1)^{m-1}}}}$.
\end{corollary}

\subsection{Conditions on $K$,\ $L$ and $m$ from continued fractions}
A consequence of Corollary \ref{LaubK} is the availability of a sharp lower and upper bound for $\frac{K}{L}$.
Suppose an $m$-cycle exists for $P(a,b,c,d)$ with $a<c,\ b>d$. Recall that $\alpha=\log cd-\log ab>0$ and $\beta=\log b -\log d>0$.
From Corollary \ref{LaubK} we find
\[
If\ F(K,L)>1 \ then\ 0 <\alpha K - \beta L< \epsilon K.\ \
If\ F(K,L)<1 \ then\ -\epsilon K < \alpha K - \beta L < 0.
\]

This means that $\frac{K}{L}$ must be a good approximation to $\rho=\frac{\beta}{\alpha}>0$.
Then Corollary \ref{LaubK} implies
\begin{corollary}\label{LaubL}\ \\
If $F(K,L)>1$ then $0<\Lambda< \epsilon \frac{\beta}{\alpha-\epsilon}L$ \\
If $F(K,L)<1$ then $0<-\Lambda< \epsilon \frac{\beta}{\alpha}L < \epsilon \frac{\beta}{\alpha-\epsilon}L$
\end{corollary}
and Corollary \ref{UbKLm} implies
\begin{corollary}\label{UbLm}\ \\
If an $m$-cycle exists for $P(a,b,c,d)$ then
 $ |\Lambda| < \frac{\beta}{\alpha-\epsilon}L \dfrac{(1+\epsilon)(ab-a-1)}{\left(\frac{d^{\frac{L}{m}}}{\beta_1}\right)^{\frac{1}{(\delta_1+1)^{m-1}}}}$.
\end{corollary}
\ \\
The bound for $\Lambda$ in Corollary \ref{UbLm} is for fixed $m$ a negative exponential function of $L$.
For large $L$ this bound is smaller than $L^{-(2+\delta)}$ for some $\delta>0$. So from Roth's lemma \cite{HS} it follows that $P(a,b,c,d)$ has for fixed $m$ a finite number of $m$-cycles. This proves the existence part of theorem \ref{mth}$(1)$.\\

We denote by $ \frac{p_n}{q_n} $ the $ n $th convergent to $ \rho \in \mathbb{R}$ with partial quotient $a_n$.
We have the following results \cite[Chapter 10]{HW}).

\begin{lemma} \label{CF} \mbox{} \\
\textnormal{\textbf{(a)}} If $ \displaystyle \frac{p}{q} $ is a rational approximation to $ \rho $ satisfying
$ | p - q \rho | <  \frac{1}{2q} $, then $  \frac{p}{q} $ is a convergent. \\
\textnormal{\textbf{(b)}} $ | p_n - q_n \rho | >  \frac{1}{q_n+q_{n+1}} >
\frac{1}{(a_{n+1}+2)q_n} $. \\
\textnormal{\textbf{(c)}} If $  \frac{p}{q} $ is a rational approximation to $ \rho $,
and if $ q \leq q_n $, then $ | p - q \rho | \geq | p_n - q_n \rho | $. \\
\textnormal{\textbf{(d)}} If $ n $ is odd then $ p_n - q_n \rho > 0 $; if $ n $ is even then $ p_n - q_n \rho < 0 $.
\end{lemma}

Using Corollary \ref{LaubL}  we  derive like Crandall \cite{Cr} a lower bound for $ L $.

\begin{lemma} \label{GenCr} \mbox{} \\
If $ q_n + q_{n+1} \leq  \left(\frac{\alpha(\alpha-\epsilon)}{\beta}\right) \frac{X_0-ab+a+1}{ab-a-1}\frac{1}{L} $, then $ L > q_n $.
\end{lemma}
\begin{proof}
Assume $ L \leq q_n $. By lemma \ref{CF}(c) and (b)
\[ |\Lambda| = | \alpha K - \beta L | =  \alpha | K-\rho L | \geq \alpha |p_n-\rho q_n| > \frac{\alpha}{q_n+q_{n+1}}  \geq   \frac{\beta}{\alpha-\epsilon}\frac{ab-a-1}{X_0-ab+a+1}L= \epsilon \frac{\beta}{\alpha-\epsilon}L \]
which contradicts Corollary \ref{LaubL}.
\end{proof}
\ \\
Lemma \ref{GenCr} implies
\begin{lemma}\label{LobLCr}
If $q_n+q_{n+1} \le \left(\frac{\alpha (\alpha-\epsilon)}{\beta}\right)\frac{X_0-ab+a+1}{ab-a-1}\frac{1}{q_n}$ then $L > q_n$
\end{lemma}
\begin{proof}
Assume that $L\le q_n$. Then $\frac{1}{q_n} \le \frac{1}{L}$ and $q_n+q_{n+1} \le \left(\frac{\alpha (\alpha-\epsilon)}{\beta}\right)\frac{X_0-ab+a+1}{ab-a-1}\frac{1}{q_n} \le \left(\frac{\alpha(\alpha-\epsilon)}{\beta}\right) \frac{X_0-ab+a+1}{ab-a-1}\frac{1}{L}$.
 Then from lemma \ref{GenCr} it follows that $L>q_n$ which contradicts the assumption $L \le q_n$ .
\end{proof}
\ \\
Applying lemma \ref{LobLCr}  gives a lower bound for $L$ as a function of $X_0$
\ \\

\begin{tabular}{|r|r|r|}
\hline $\log_{10}X_0$ & $L(1,3,2,2)>$ & $L(2,4,3,3)>$  \\
\hline \hline
3 & 5 & 2 \\ \hline
4 & 17 & 2 \\ \hline
5 & 22 & 9 \\ \hline
6 & 127 & 52 \\ \hline
7 & 276 & 52 \\ \hline
8 & 276 & 113 \\ \hline
9 & 6475 & 113 \\  \hline
10 & 13226 & 2651\\ \hline \hline
\end{tabular}

\subsection{Application of transcendence theory}

Transcendence theory shows that
linear forms in logarithms of integers cannot be too small in terms of their coefficients.
The original paper of Baker \cite{Ba} gives a lower bound for a linear form in $n$ logarithms. Laurent a.o. \cite{LMN} give a  bound for two logarithms and
Rhin \cite{Rh} derived a sharper lower bound for the specific
case $ x \log 2 + y \log 3 $.
The result of Baker supplies a lower bound for $\Lambda$ of $P(a,b,c,d)$. Corollary \ref{UbLm} supplies an upper bound for $\Lambda$ as a function of $L$ and $m$. For fixed $m$ the upper bound is a negative exponential function of $L$ which for large $L$ is smaller than the lower bound of Baker and this latter bound is polynomial in $L$. So for each $m$ the upper bound for $L$ can be computed. This proves the computability  part of theorem \ref{mth}$(1)$.

\section{Cycle existence}
For the functions $P(2,2,1,3)$ and $P(2,4,3,3)$ Rhin's lower bound is applicable. We calculated theoretical and numerical bounds for cycles up to an arbitrary chosen upper bound for $m$. Recall that $\alpha= \log(cd)-\log(ab),\ \beta=\log(b)-\log(d)$.

\subsection{Numerical results for $P(2,2,1,3)$}
 Because of the constraint $\alpha,\beta>0$ we analyze $P(1,3,2,2)$ which has the same trajectories.  Let $X_0=10^6$. From Rhin  \cite{Rh} we derive the following estimate  with $\epsilon= \frac{1}{X_0-1}$.

\begin{lemma} \label{Rhin1}
\[ | \Lambda | > e^{- 13.3 ( 1.34 + \log (L) ) } . \]
\end{lemma}
\begin{proof}
We apply Rhin's proposition on p.\ 160 with $u_0=0,\ u_1=2K+L,\ u_2=-(K+L)$. Then $ H = u_1 = 2K + L $ and Rhin's estimate
leads to   $ | \Lambda | \ge [2K+L]^{-13.3}$. From Corollary \ref{LaubK} we have: if $\Lambda>0$ then $0< \alpha K - \beta L < \epsilon K $ which implies $K < \frac{\beta}{\alpha-\epsilon}L$ and if $\Lambda<0$ then $0< -\alpha K + \beta L < \epsilon K $  which implies $K < \frac{\beta}{\alpha}L$. So in both cases we have $K < \frac{\beta}{\alpha-\epsilon}L$ and we find
 $ | \Lambda |  \ge [2K+L]^{-13.3} > \left(\frac{\alpha+2\beta-\epsilon}{\alpha-\epsilon}L\right)^{-13.3}> e^{-13.3(1.34 + \log (L))}$.
\end{proof}
\ \\
Let $x=x_3(m)$ be the solution of $e^{- 13.3 ( 1.34 + \log (x) ) }=  \frac{\beta x}{\alpha-\epsilon} \frac{(1+\epsilon)(ab-a-1)}{\left(\frac{d^{\frac{x}{m}}}{\beta}\right)^{\frac{1}{(\delta+1)^{m-1}}}}$. Then $L \le x_3(m)$.
We calculated the cross-over point $L=x_3(m)$ as a function of $m$
\ \\

\begin{tabular}{|r|r|}
\hline $m$ & $L\le$  \\
\hline \hline
1 & 126 \\ \hline
2 & 1241 \\ \hline
3 & 8171 \\ \hline
4 & 45588 \\ \hline
5 & $2.3201 \times 10^5$ \\ \hline
10 & $4.2643 \times 10^8$ \\ \hline
20 & $4.9668 \times 10^{14}$ \\ \hline
50 & $1.1449  \times 10^{32}$ \\ \hline
100 & $2.2665  \times 10^{60}$ \\ \hline \hline
\end{tabular}

\ \\

We now apply an upper bound reduction technique. Let $x=x_1(m)$ be the solution of \\
$e^{- 13.3 ( 1.34 + \log (x) ) }= \frac{\alpha}{2x}$. Lemma \ref{CF} implies that if $L\ge L_1=x_1(m)$ then $\frac{K}{L}$ must be a convergent to $\frac{\beta}{\alpha}$. The first convergents are \\
\mbox{ $(3,2),(7,5),(24,17),(31,22),(179,127),(389,276),(9126,6475),(18641,13226),(46408,32927) \cdots$}. For the $L$-interval up to $x_3(m)$ for $ m\le 20$ we found that the maximum partial quotient is $55$. Let $x=x_2(m)$ be the solution of $e^{- 13.3 ( 1.34 + \log (x) ) }= \frac{\alpha}{57x}$. Then lemma \ref{CF} implies that $L\le L_2=x_2(m)$. Applying these bounds as a function of $m$ gives an interval for $L$ being a convergent solution  of a potential $m$-cycle.
\ \\

\begin{tabular}{|r|r|r|}
\hline $m$ & $L \ge L_1$ & $L \le L_2$ \\
\hline \hline
1 & 10 & 16 \\ \hline
2 & 122 & 162 \\ \hline
3 & 875 & 1085 \\ \hline
4 & 5120 & 6103 \\ \hline
5 & 26893 & 31240 \\ \hline
10 & $5.3270 \times 10^7$ & $5.8249 \times 10^7$ \\ \hline
20 & $6.5093 \times 10^{13}$ & $6.8249 \times 10^{13}$ \\ \hline \hline
\end{tabular}

\ \\

For $m=1$ the lower bound $127$ from lemma \ref{GenCr} is larger than the upper bound $L_2$ in this table. For $m=2,3,4,5$ there are no convergents in the interval $(L_1,L_2)$. We checked for $L<31240$ if corollary \ref{LaubL} is satisfied. This is true for $(K,L)$ is $(389,276),(778,552),(957,679),(1167,828)$ and many more pairs with $L>828$. We computed all trajectories for starting values $<X_0=10^6$ until either a cycle appeared or an apparent divergent trajectory exceeded $10^{8}$ with $m>10$. From these calculations we find

\begin{corollary}\label{cycinvcol}
The Collatz permutation \\
(a) has for  $m\le 2$ no other cycles than $(1),\ (2,3),\ (4,6,9,7,5)$.\\
(b) has for  $m=3$ the cycle $(44,66,99,74,111,83,70,105,79,59)$ and if $x_i\ge 10^6$ could have cycles with $(K,L)\ =\ $ $(389,276),(778,552),(957,679),(1167,828)$. \\
(c) has for $m=4$ no cycles with $x_i<10^6$ and if $x_i \ge 10^6$ could have cycles with $(K,L)$ $=$ $(389,276),(778,552),(957,679),(1167,828)$ and $828 < L < 5120$. \\
(d) has for $m=5$  no cycles with $x_i<10^6$ and if $x_i \ge 10^6$ could have cycles with $(K,L)$ $=$ $(389,276),(778,552),(957,679),(1167,828)$ and $828 < L < 26893$.
\end{corollary}
\ \\
This proves theorem \ref{mth}$(3)$

\subsection{Numerical results for $P(2,4,3,3)$}
Let $a=2,b=4,c=3,d=3,X_0=10^6$. From Rhin \cite{Rh} we derive
\begin{lemma} \label{Rhin2}
\[ | \Lambda | > e^{- 13.3 ( 1.77 + \log (L) ) } . \]
\end{lemma}
\begin{proof}
We apply Rhin's proposition on p.\ 160 with $u_0=0,\ u_1=2K+L,\ u_2=-(3K+2L)$. Then $ H = u_2 = 3K + 2L $ and Rhin's estimate
leads to   $ | \Lambda | \ge [3K+2L]^{-13.3}$. From Corollary \ref{LaubK} we have: if $\Lambda>0$ then $0< \alpha K - \beta L < \epsilon K $ which implies $K < \frac{\beta}{\alpha-\epsilon}L$ and if $\Lambda<0$ then $0< -\alpha K + \beta L < \epsilon K $  which implies $K < \frac{\beta}{\alpha}L$. So in both cases we have $K < \frac{\beta}{\alpha-\epsilon}L$ and we find
 $ | \Lambda |  \ge [3K+2L]^{-13.3} > \left(\frac{2\alpha+3\beta-2\epsilon}{\alpha-\epsilon}L\right)^{-13.3}> e^{-13.3(1.77 + \log (L))}$.
\end{proof}
\ \\
Let $x=x_3(m)$ be the solution of $e^{- 13.3 ( 1.77 + \log (x) ) }=  \frac{\beta x}{\alpha-\epsilon} \frac{(1+\epsilon)(ab-a-1)}{\left(\frac{d^{\frac{x}{m}}}{\beta}\right)^{\frac{1}{(\delta+1)^{m-1}}}}$. Then $L \le x_3(m)$.
We calculated the cross-over point $L=x_3(m)$ to find
\ \\

\begin{tabular}{|r|r|}
\hline $m$ & $L\le$  \\
\hline \hline
1 & 88 \\ \hline
2 & 754 \\ \hline
3 & 4422\\ \hline
4 & 22142 \\ \hline
5 & $1.0150 \times 10^5$ \\ \hline
10 & $1.1314 \times 10^8$ \\ \hline
20 & $5.0142 \times 10^{13}$ \\ \hline
50 & $  6.6686\times 10^{29}$ \\ \hline
100 & $ 1.1640  \times 10^{56}$ \\ \hline \hline
\end{tabular}

\ \\

We now apply an upper bound reduction technique. Let $x=x_1(m)$ be the solution of \\
$e^{- 13.3 ( 1.77 + \log (x) ) }= \frac{\alpha}{2x}$. Lemma \ref{CF} implies that if $L\ge L_1=x_1(m)$ then $\frac{K}{L}$ must be a convergent to $\frac{\beta}{\alpha}$. The first convergents are \\ \mbox{ $(5,2),(17,7),(22,9),(127,52),(276,113),(6475,2651),(13226,5415),(32927,13481),\cdots$}.  For the $L$-interval up to $x_3(m)$ for
$ m\le 20$ we found that the maximum partial quotient is $55$. Let $x=L_2=x_2(m)$ be the solution of $e^{- 13.3 ( 1.77 + \log (x) ) }= \frac{\alpha}{57x}$. Then lemma \ref{CF} implies that $L\le x_2(m)$. We took as a numerical lower bound for cycles the value $X_0=10^6$. Applying these bounds as a function of $m$ gives an interval for $L$ being a convergent solution  of a potential $m$-cycle.
\ \\

\begin{tabular}{|r|r|r|}
\hline $m$ & $L \ge L_1$ & $L \le L_2$ \\
\hline \hline
1 & 9 & 12 \\ \hline
2 & 84 & 107 \\ \hline
3 & 517 & 625 \\ \hline
4 & 2661 & 3124 \\ \hline
5 & 12437 & 14307 \\ \hline
10 & $1.4561 \times 10^7$ & $1.5903 \times 10^7$ \\ \hline
20 & $6.676 \times 10^{12}$ & $7.034 \times 10^{13}$ \\ \hline \hline
\end{tabular}

\ \\

For $m=1$ the lower bound $52$ from lemma \ref{GenCr} is larger than the upper bound $L_2$ in this table.  For $m=2,3,4$ there are no convergents in the interval $(L_1,L_2)$. For $m=5$ there is one convergent ($K=32927,L=13481$) in the interval $(L_1,L_2)$. We computed all trajectories for starting values $<X_0=10^6$ until either a cycle appeared or an apparent divergent trajectory exceeded $10^{8}$ with $m>20$. From these calculations we find

\begin{corollary}\label{cycps2433}
The permutation $P(2,4,3,3)$ \\
(a) has for $m=1,2$ no other cycles than listed in the table below.\\
(b) has for $m=3$ the $3$-cycle listed in the table below and if $x_i \ge 10^6$ could have cycles with $55<L<517$.\\
(c) has for $m=4$ no cycles with $x_i<10^6$ and if $x_i \ge 10^6$ could have cycles with $55<L<2661$.\\
(d) has for $m=5$ no cycles with $x_i<10^6$ and if $x_i \ge 10^6$ could have cycles with $55<L<12437$ and with $K=32927,L=13481$.\\
\end{corollary}

Numerically found cycles for $P(2,4,3,3)$.
\ \\

\begin{tabular}{|r|r|r|r|r|}
\hline $nr$ & $x_{min}$ & $x_{max}$ &$length$ & $m$  \\
\hline \hline
1 & 1 & 1 &  1 & 0 \\ \hline
2 & 2 & 2 & 1 & 0 \\ \hline
3 & 3 & 4 & 2 & 1 \\ \hline
4 & 5 & 5 & 1 & 0 \\ \hline
5 & 6 & 8 & 3 & 1 \\ \hline
6 & 9 & 16 & 7 & 1 \\ \hline
7 & 15 & 32 & 14 &3 \\ \hline
8 & 27 & 176 & 51 &10 \\ \hline
9 & 33 & 52 & 7 & 2 \\ \hline
10 & 90 & 1972 & 93 & 19 \\ \hline
11 & 213 & 700 & 31 & 7 \\ \hline
12 & 645 & 1612 &  31 & 8 \\ \hline \hline
\end{tabular}

\ \\

This result  proves theorem \ref{mth}$(5)$.

\section{Generalizations of $P(a,b,c,d)$}

\subsection{Simple generalizations}
The definition of $P(a,b,c,d)$ guarantees that only elements $\equiv 0$ (mod $b$) occur in decreasing subsequences in between a local maximum and a local minimum. If the residues are in arbitrary order and we consider elements  with $n>\frac{ab-2}{cd-ab}$, then the worst case $abn+ab-1 \rightarrow cdn+1$  still results in an increasing subsequence. So there exist $(ab-a)!-1$ functions, each being a simple generalization of $P(a,b,c,d)$ for which the developed theory and calculation of $m$-cycles is applicable. This proves theorem \ref{mth}$(2)$.\\

The Collatz permutation has $2!-1=1$ simple generalization.
We computed all trajectories for starting values $<X_0=10^6$ until either a cycle appeared or an apparent divergent trajectory exceeded $10^8$ with $m>20$. We found the following results:
\ \\

\begin{tabular}{|r|r|r|r|r|}
\hline $nr$ & $x_{min}$ & $x_{max}$ &$length$ & $m$  \\
\hline \hline
1 & 1 & 3 &  3 & 1 \\ \hline
2 & 4 & 27 & 11 & 2 \\ \hline
3 & 5 & 5& 1 & 0 \\ \hline
4 & 10 & 15 & 2& 1 \\ \hline
5 & 14& 21& 3 & 1 \\ \hline
6 & 16& 261 & 34 & 8 \\ \hline
7 & 20 & 45 & 5 &1 \\ \hline
8 & 220 & 555 & 12 & 4 \\ \hline \hline
\end{tabular}

\ \\
This result cannot be obtained through the SdW approach. This result and the potential cycles in Corollary \ref{cycinvcol} prove theorem \ref{mth}$(4)$.\\

For the second example $P(2,4,3,3)$ there are  $6!-1=719$ simple generalizations. We computed for the first (ordered by permutation number) $150$ simple generalizations all trajectories for starting values $<X_0=10^5$ until either a cycle appeared or an apparent divergent trajectory exceeded $10^8$ with $m>20$. See the section \textbf{Remarks $2$}. There is a "large" variety for the parameters of the cycle structure per function: number of cycles  from $9$ to $23$,\ maximal cycle length from $11$ to $1291$,\ maximal element from $124$ to $ 30\ 010\ 496$. Potential cycles in Corollary \ref{cycps2433} could exist for all these simple generalizations too.

\subsection{Extended generalizations}
If also the first residue class changes, there is no fixed ratio $\frac{cd}{ab}$ for the increasing subsequences. Consequently the developed theory for cycle existence cannot simply be applied. See the section \textbf{Extending the class of permutations}.

For the Collatz permutation, there are $4$  extended generalizations.  We computed trajectories with a starting value $<X_0=10^6$ until a cycle appeared or an apparent divergent trajectory exceeded $10^8$.
We found the following results. The Collatz permutation and its simple generalization are included in the table below for comparison.
\ \\

\begin{tabular}{|r|c|l|}
\hline $f\ nr$ & $\#\ of\ cycles$ & $(x_{min},x_{max},length,m)$ \\
\hline \hline
1 & 4 & (1,1,1,0),\ (2,3,2,1),\ (4,9,5,2),\ (44,111,12,4)   \\ \hline
2 & 8 & (1,3,3,1),\ (4,27,11,2),\ (5,5,1,0),\ (10,15,2,1), \\
  &   &  (14,21,3,1),\ (16,261,34,8),\ (20,45,5,1),\ (220,555,12,4) \\ \hline \hline
3 & 3 & (0,1,2,1),\ (2,7,5,1),\ (42,109,12,4)  \\ \hline
4 & 2 & (0,5,5,1),\ (40,107,12,4)  \\ \hline
5 & 6 & (0,7,6,1),\ (5,5,1,0),\ (12,19,2,1), \\
  &   &  (26,61,5,1),\ (140,5215,94,26),\ (306,775,12,4) \\ \hline
6 & 7 &  (1,1,1,0),\ (0,23,11,2),\ (6,11,1,1),\ (10,17,3,1), \\
  &   &  (12,257,34,8),\ (16,41,5,1),\ (216,551,12,4) \\ \hline \hline
\end{tabular}

\ \\

For the second example $P(2,4,3,3)$ there are $4320$ extended generalizations. We computed  trajectories with a starting value $<X_0=10^5$ until a cycle appeared or an apparent divergent trajectory exceeded $10^8$ for the first $150$ extended generalizations. We found a "large" variety in the parameters for the cycle structure per function:\  number of cycles from $3$ to $10$, maximal cycle length from $2$ to $118$, maximal element from $14$ to $214\ 649$.

\section{Divergent trajectories}

\subsection{Introduction}
 The conjecture for the Collatz sequence is that all trajectories end in the cycle $(1,2)$.  For the Collatz permutation sequences the analogous  conjecture is that all trajectories are either one of the $4$ known cycles or divergent. Numerical evidence and transcendence theory lead to the common conjecture of a finite number of cycles. For the Collatz sequence this  implies that divergent trajectories can exist, for the Collatz permutation sequences this implies that divergent trajectories must exist. Indeed we found (no proof) that the number of apparent divergent trajectories  increases linear with the range of starting values $X_0$. For $P(1,3,2,2)$ the ratio is  $.05X_0$ and for $P(2,4,3,3)$ the ratio is $.12X_0$ approximately.

\subsection{Distribution of elements in divergent trajectories of $P(a,b,c,d)$}
For the numbers in a divergent trajectory of the Collatz permutation sequence Guy \cite{Gu} mentions an intriguing paradox. Assuming that numbers in such a trajectory are uniformly distributed (mod $2$) and (mod $3$), then  from left to right the average multiplication factor is $\sqrt{\frac{3}{2}\frac{3}{4}} \approx 1.06066$ and from right to left the average multiplication factor is $\sqrt[3]{ \frac{2}{3} (\frac{4}{3} )^2} \approx 1.05827$.  These factors should be reciprocal.

 A divergent trajectory has two branches: a left branch downwards to the minimal element and a right branch from the minimal element upwards. Numerical evidence suggests that for the Collatz permutation sequence the distribution of elements in a branch of a divergent trajectory is uniform either (mod $b$) or (mod $d$). The next lemma proves that for $P(a,b,c,d)$  a uniform distribution of elements (mod $b$) and (mod $d$) in a branch of a divergent trajectory cannot exist.

\begin{lemma}
Assume that $P(a,b,c,d)$ has divergent trajectories.  Then in a branch at least one of the distributions of elements (mod $b$) and (mod $d$) is non-uniform.
\end{lemma}
\begin{proof}
From left to right in the decreasing part, let $P(x \equiv 0$ (mod $b$)$)=\alpha$ and let $P(x \equiv 0$ (mod $d$)$)= \beta$. From left to right we have in the decreasing part of the trajectory a multiplication factor $f_1=(\frac{d}{b})^{\alpha}(\frac{cd}{ab})^{1-\alpha}= (\frac{c}{a})^{1-\alpha}\frac{d}{b}$. From right to left we have in the increasing part a multiplication factor $f_2=(\frac{b}{d})^{\beta}(\frac{ab}{cd})^{1-\beta}= (\frac{a}{c})^{1-\beta}\frac{b}{d}$. By definition $f_1.f_2=1$ from which follows  $\alpha=\beta$. A uniform distribution (mod $b$) respectively (mod $d$) implies $\alpha= \frac{1}{b}$ and $\beta=\frac{1}{d}$. So at least one of the distributions is non-uniform. A similar reasoning holds for the from left to right increasing part of the divergent trajectory.
\end{proof}

\subsection{Theoretical conditions for divergent trajectories}
Let $d>2$. Let $x_i,\ r_0,\ldots,r_{d-1}\ \in \mathbb{Z}$, $m_0,\dots,m_{d-1} \in \mathbb{N}$  with $r_i \equiv im_i$ (mod $d$). Matthews \cite{Ma} introduced the generalized Collatz function $T:\ \mathbb{Z}\rightarrow \mathbb{Z}$.
\[
x_{i+1}=  T(x_i)=\frac{m_ix_i-r_i}{d}  \mbox{ \ if \ } x_i \equiv i \pmod d
\]
Matthews conjectures (i) if $\prod_{i=0}^{d-1}m_i<d^d$ then all trajectories end in a cycle (ii) if $\prod_{i=0}^{d-1}m_i>d^d$ then almost all trajectories are divergent.
Levy \cite{Le} (his theorem $18$) proves that if $\sum_{i=1}^{k}\frac{1}{m_i} \le 1$ then $\prod_{i=0}^{d-1}m_i \ge d^d$. The equality holds if and only if $m_0= \cdots =m_{d-1}=d$.
\ \\
Let $C$ be a  set of $k \ge 2$ ordered pairs $\{(a_i,b_i),i=1,\cdots,k\}$ with  $a_i\in \mathbb{N}$ and  $b_i \in  \mathbb{N}_0$.
We call $C$ a complete coverage set (CC set) if for every x $\in \mathbb{N}$ there exists exactly one $i$ for which $x \equiv b_i$ (mod $a_i$). Examples are $\{(3,0),(3,1),(3,2)\}$,\ $\{(2,0),(4,1),(8,3),(16,7),(16,15)\}$. See the section \textbf{Remarks $3$}.

\begin{lemma}\label{nscps}
 Let $\{(a_i,b_i),i=1,\cdots,k\}$ be a CC-set. Then $\sum_{i=1}^k \frac{1}{a_i}=1$.
\end{lemma}
\begin{proof}
Let $A=\lcm{a_i}$ and let $\gamma_i=\frac{A}{a_i}$. \\
The set $\{(A,b_1),(A,b_1+a_1), \ldots ,(A,b_1 +(\gamma_1-1)a_1), \ldots ,
(A,b_k), (A,b_k+a_k), \ldots , (A,b_k+(\gamma_k-1)a_k)\}$ is a CC-set.
 Suppose that $x \equiv b_i$ (mod $a_i$). Then $x=b_i+m.a_i$. Let $m=r+s\gamma_i$. Then $0 \le r < \gamma_i$ and
$x=b_i+ra_i+s \gamma_i a_i$. Hence $x \equiv b_i+ra_i$ (mod $\gamma_ia_i$) $\equiv b_i+ra_i$ (mod $A$).  Now suppose that $x \equiv b_i+r_ia_i$ (mod $\gamma_ia_i$) and $x \equiv b_j+r_ja_j$ (mod $\gamma_ja_j$). Then $x \equiv b_i$ (mod $a_i$) and $x \equiv b_j$ (mod $a_j$) which contradicts that
$\{(a_i,b_i),i=1,\cdots,k\}$ is a CC-set. Hence the set $\{(A,b_1),(A,b_1+a_1), \ldots ,(A,b_1 +(\gamma_1-1)a_1), \ldots ,
(A,b_k), (A,b_k+a_k), \ldots , (A,b_k+(\gamma_k-1)a_k)\}$ is a CC-set. \\
This CC-set has $\sum_{i=1}^k \gamma_i$ elements. Let $(a_j^*,b_j^*)$ be an element of this CC-set. Because of the constant $a_j^*=A$ the set $\{b_j^*\}$ is a complete residue system (mod $A$). Hence the set $\{b_j^*\}$ has $A$ elements. We find that
$\sum_{i=1}^k \gamma_i=A$ from which we conclude $\sum_{i=1}^k \frac{1}{a_i}=1$.
\end{proof}
\ \\
A permutation can be written as a generalized Collatz function. Permutations are exactly those generalized Collatz functions for which Matthews does not conjecture convergence or divergence of trajectories.

\subsection{Permutations with $d \equiv 0$ (mod $b$)}
If $b|d$ then a special structure of trajectories appears. Let $b=v.d$ with $v>1$ and let $\gcd(b-1,d-1)=w \ge 1$. Then $a(b-1)=a.w.(b^*-1)=c(d-1)=c.w.(d^*-1)$. Consequently $a=u(d^*-1)$ and $c=u(b^*-1)$ and we find the permutation

\[
\left.
\begin{array}{c}
v.d.n \\
u(d^*-1)v.dn +s_i \ ( s_i \not\equiv 0 \pmod{v.d})
\end{array}
\right\}
\leftrightarrow
\left\{
\begin{array} {c}
 d.n \\
u(b^*-1)d.n +r_i \ (r_i \not\equiv 0 \pmod{d})
\end{array}
\right.
 \]

Recall that $b>d$ implies $ab<cd$. Let $x_i \ge v.d$. If $x_i \not\equiv 0$ (mod $v.d$) then $x_{i+1}>x_i$. Hence
$x_{i+1} \not \equiv 0$ (mod $d$) and $x_{i+1} \not \equiv 0$ (mod $v.d$). So $x_i$ is the start of a divergent trajectory. If $x_i \equiv 0$ (mod $v.d$) then $x_{i+1}<x_i$ and as long as $x_{i+1} \equiv 0$ (mod $v.d$) the decreasing trajectory continues. Apart from trivial $1$-cycles for $0 \le x_i<v.d$ this permutation has infinitely many divergent trajectories. Consider as an example $P(2,6,5,3)$ with $u=1$
\[
\left.
\begin{array}{c}
6n \\
12n+1,2,3,4,5,7,8,9,10,11
\end{array}
\right\}
\leftrightarrow
\left\{
\begin{array} {c}
 3n \\
15n+1,2,4,5,7,8,10,11,13,14
\end{array}
\right.
 \]
For this permutation there are three trivial $1$-cycles $(0),\ (1),\ (2)$ and an infinite number of divergent trajectories with minima $3,\ 9,\ 15,\ \cdots \ $. This proves theorem \ref{mth}$(6)$. For simple generalizations the same analysis is applicable, however for $n \le \frac{ab-2}{cd-ab}$ there can also exist  $m$-cycles with $m>1$. For extended generalizations the existence of divergent trajectories cannot be proved. This example shows the limitations of our approach that can only prove that for each $m$ a finite number of $m$-cycles exists.

\section{Extending the class of permutations}
The class of permutations in lemma \ref{gps} with $a(b-1)=c(d-1)$ can be extended, such that our method
of finding cycles is partly applicable.

\begin{lemma}\label{fafc}
Let $a,b,c,d\in \mathbb{N}_0$, satisfying $a>0,\ c>0,\ b>1,\ d>1$,\ $\gcd(a,c)=\gcd(b,d)=1$. Let $1\le f_a|a, 1\le f_c|c$ and let  $N=a(b-1)+f_a=c(d-1)+f_c$. Let $R = \{r_i, i=1 \cdots N\}$ be the set of different residues (mod $ab$) with
$r_i \not\equiv 0$ (mod $b$) and let $S = \{s_i,i=1 \cdots N\}$ be the set of different  residues (mod $cd$) with $s_i \not\equiv 0$ (mod $d$). If  $1 \le f_a \le a,\ 1 \le f_c \le c$ then the function $f$ defined by\\
\[
\left.
\begin{array}{c}
f_a.bn \\
f_a.bn + b \\
\cdots \\
f_a.bn +(f_a-1)b\\
abn +r_i \ \ r_i \not\equiv 0 \pmod b
\end{array}
\right\}
\leftrightarrow
\left\{
\begin{array} {c}
f_c.dn \\
f_c.dn +d \\
\cdots \\
f_c.dn+(f_c-1)d \\
cdn +s_i \ \ s_i \not\equiv 0 \pmod d
\end{array}
\right.
 \]
is a permutation.
\end{lemma}
\begin{proof}
We distinguish five cases
\begin{enumerate}
  \item The case $f_a=f_c=1$. Then $a,b,c,d$ satisfy the conditions of lemma \ref{gps}.
  \item The case $f_a=1$,\ $f_c=c>1$. Now $ab-a+1=cd$. Since $ab\ne cd$ it follows that $a\ne1$ and we find the function
  \[
  \left.
  \begin{array}{c}
  bn\\
  ab.n+r_i \ \ r_i \not\equiv 0 \pmod b
  \end{array}
  \right\}
  \leftrightarrow
  cd.n + r_j \ \ r_j \in \{0,1, \cdots ,cd-1\}
  \]
  This function is invariant under the substitution $c \rightarrow c*=1,\ d \rightarrow d*=cd,\ f_c \rightarrow f_{c*}=1$. $P(a,b,c*,d*)$ satisfies lemma \ref{gps}.
  \item The case $f_a=a>1$,\ $f_c=1$. After the transformation $a \rightarrow c,\ b \rightarrow d$ this case is identical to the case $f_a=1$,\ $f_c=c>1$.
  \item The case $f_a=a$,\ $f_c=c$. Now $ab=cd$ and $f$ becomes the trivial function $n \leftrightarrow n$.
  \item The case $1< f_a<a$,\ $1< f_c<c$. Now the conditions of lemma \ref{gps} are not satisfied.
  For $a=10,b=8,f_a=5,c=9,d=9,f_c=3$ we find the permutation
\[
\left.
\begin{array}{c}
40n, 40n+8, 40n+16, 40n+24, 40n+32 \\
80n + 1\cdots 7, \cdots ,80n+73 \cdots 79
\end{array}
\right\}
\leftrightarrow
\left\{
\begin{array} {c}
27n, 27n+9, 27n+18 \\
81n + 1 \cdots 8, \ldots ,81n+73 \cdots 80
\end{array}
\right.
 \]
\end{enumerate}
\end{proof}

\ \\
For permutations of lemma \ref{fafc} our method to find cycles and divergent trajectories must be adapted in a nontrivial way, since increasing subsequences do not have a fixed ratio $\frac{cd}{ab}$ anymore. In general a linear form in three or more logarithms appears. Then Baker's result to find an upper bound for the cycle length is still applicable, however reducing the upper bound becomes nontrivial.\\

Let ${a_i,b_i,c_i,d_i}|i=1,\cdots N \in \mathbb{N}_0$. Consider the function  $f\ :\ \mathbb{N}_0 \rightarrow \mathbb{N}_0$ defined by
\[
\left.
\begin{array}{c}
a_1n+b_1 \\
a_2n+b_2\\
\cdots \\
a_Nn+b_N
\end{array}
\right\}
\leftrightarrow
\left\{
\begin{array} {c}
c_1n+d_1\\
c_2n+d_2\\
\cdots \\
c_Nn+d_N
\end{array}
\right.
 \]

A necessary and sufficient condition for $f$ being a  permutation is that the sets $\{(a_i,b_i),i=1 \cdots N\}$ and $\{(c_i,d_i),i=1 \ldots N\}$ are CC-sets. Venturini \cite{Ve} chooses $a_i=pq,b_i=i$ and proves a similar result for permutations from $\mathbb{Z}$ onto $\mathbb{Z}$.\\

There exist permutation functions for which the CC-set condition does not apply. Let $P(n)$ be the $n^{th}$ prime number and let $C(n)$ be the $n^{th}$ composite number. Then the function
\begin{eqnarray*}
    1 & \leftrightarrow & 1 \\
   2n & \leftrightarrow & P(n) \\
   2n+1 & \leftrightarrow & C(n)
 \end{eqnarray*}
is a permutation with (at least) the cycles \\
$(1),\ (2),\ (3,4),\ (5,6),\ (7,8),\ (9),\ (10,11),\ (12,13),\ (14,17,15)$, a cycle $(18,\cdots)$ with length $22$,
a cycle $(62,\cdots)$ with length $3$, a cycle $(84,\cdots)$ with length $3$, and a cycle $(92,\cdots)$ with length $6$. There are numerically "divergent", however theoretically unknown trajectories.

\section{Remarks}
\begin{enumerate}
\item
The Collatz function $g(2n)=n,\ g(2n+1)=3n+2$ is well known \cite{La}. The Collatz function is not a bijection from $\mathbb{N}$ onto $\mathbb{N}$. The Collatz conjecture states  that for all $n>0$ finally the cycle $(1,2)$ appears.  Collatz also introduced the Collatz permutation and used the slightly different notation $f(2n)=3n,f(4n+1)=3n+1,f(4n-1)=3n-1$. Known cycles and conjectured divergent trajectories are the same, however now if $n\in\mathbb{N}_0$ also the cycle $(-1)$ appears. The Collatz permutation has different names in the literature. Lagarias and  Tavares \cite{La,Ta} call it \emph{Collatz's original problem} and Guy \cite{Gu} calls it \emph{probably the inverse of Collatz's original problem}.

\item
Cycle computation time with Mathematica can be quite large on a standard notebook with dual processor Intel(R) Celeron(R) CPU $1007$U, $1.50$GHz.  Computation of the results for $P(1,3,2,2)$ took $20$ hours and
computation of the results for $P(2,4,3,3)$ took $24$ hours (See subsection \textbf{Extended generalizations}). When a larger range for starting values  is used  then lemma \ref{LobLCr} supplies a larger lower bound for the cycle length.

\item
A CC set is a concept similar (not identical) to a covering system of congruences of Erd\"{o}s \cite{Ch}: a finite set of $k$ residue classes $b_i$ (mod $a_i$) with $2\le a_i \le a_{i+1} \le a_k$ with the property that $\forall x \in \mathbb{N}$ there exists at least one $i$ for which $x \equiv b_i$ (mod $a_i$). The difference is the uniqueness of the residue class for which
$x \equiv b_i$ (mod $a_i$). It is also similar (not identical) to an exact covering system (See problem F14 in \cite{Gu}). The difference is the existence of infinite CC sets, e.g., $\{(3,0),(3,1),(3.2^{k+1},3.2^k-1)|k \ge 0\}$. Covering systems are  extensively analyzed in the literature \cite{BFF,Ch,Po,PS}.

\end{enumerate}

\subsection*{Acknowledgements}
The author expresses his gratitude to Aad Dijksma and to the anonymous reviewers for many useful remarks on the manuscript.

\end{document}